\documentclass[aap]{imsart}

\usepackage{amssymb,amsmath,amsthm,bbm}

\usepackage[latin1] {inputenc}
\usepackage{graphicx, psfrag}
\usepackage{hyperref}
\RequirePackage[numbers]{natbib}
\usepackage{txfonts}
\usepackage{mathrsfs}
\usepackage{enumerate}
\usepackage{dsfont}
\usepackage{nicefrac}	

\newtheorem{theorem}{Theorem}[section]
\newtheorem{lemma}[theorem]{Lemma}
\newtheorem{proposition}[theorem]{Proposition}

\numberwithin{equation}{section}

\theoremstyle{remark}
\newtheorem*{remark}{Remark}

\def\supp{\mathop{\mathrm{supp}}\nolimits}

\def\supp{\mathop{\rm supp}\nolimits}

\def\d{\mathrm{d}}

\def\<{\langle}
\def\>{\rangle}

\def\a{\alpha}

\def\e{\epsilon}

\def\g{\gamma}
\def\k{\kappa}

\def\s{\sigma}
\def\t{\tau}

\def\D{\Delta}

\def\del{\partial}
\def\R{{\Bbb R}}  
\def\N{{\Bbb N}}  
\def\P{{\Bbb P}}  
\def\Z{{\Bbb Z}}

\def\E{{\Bbb E}}

\let\cal=\mathcal
\def\AA{{\cal A}}

\def\EE{{\cal E}}
\def\FF{{\cal F}}

\def\PP{{\cal P}}

 \def \k {{\kappa}}

 \def \s {{\sigma}}

 \def \D {{\Delta}}
 \def \t {{\tau}}

 \def \g {{\gamma}}
 
 \def \d {{\delta}}
 \def \a {{\alpha}}

 \def \del {{\partial}}

 \def \ba {\begin{array}}
 \def \ea {\end{array}}

 \newcommand{\be}{\begin{equation}}
 \newcommand{\ee}{\end{equation}}

\newcommand{\bea}{\begin{eqnarray}}
 \newcommand{\eea}{\end{eqnarray}}
\def\TH(#1){\label{#1}}\def\thv(#1){\ref{#1}}
\def\Eq(#1){\label{#1}}\def\eqv(#1){(\ref{#1})}

 \def \1{\mathbbm{1}}
\def\wt {\widetilde}
\def\wh{\widehat}

\def\eee{\mathrm{e}}

\arxiv{arXiv:1412.5975}

\startlocaldefs
\endlocaldefs

\begin{document}

\begin{frontmatter}

\title{ Extended Convergence of the Extremal Process of Branching Brownian Motion}
\runtitle{Extended Convergence of the Extremal Process of BBM}


\begin{aug}
\author{\fnms{Anton} \snm{Bovier}\thanksref{t1,m1}\corref{Anton Bovier}\ead[label=e1]{bovier@uni-bonn.de}}
\and
\author{\fnms{Lisa} \snm{Hartung}\thanksref{t2,m1,m2}\ead[label=e2]{lisa.hartung@nyu.edu}}

\thankstext{t1}{Partially supported through the German Research Foundation in 
the Collaborative Research Center 1060    \emph{The Mathematics of Emergent Effects}, 
the Priority Programme  1590 \emph{Probabilistic Structures in Evolution},
the \emph{Hausdorff Center for Mathematics} (HCM), and  the Cluster of Excellence \emph{ImmunoSensation} at Bonn University.}
\thankstext{t2}{Supported by the German Research Foundation in the \emph{Bonn International Graduate School in Mathematics} (BIGS) and the
Collaborative Research Center 1060 \emph{The Mathematics of Emergent Effects}.}

\address{Institut f\"ur Angewandte Mathematik\\
Rheinische Friedrich-Wilhelms-Universit\"at\\ Endenicher Allee 60\\ 53115 Bonn, Germany\\
  \printead{e1}}
   \affiliation{Bonn University\thanksmark{m1} and  New York University\thanksmark{m2}}

\address{L. Hartung\\Institut f\"ur Angewandte Mathematik\\
Rheinische Friedrich-Wilhelms-Universit\"at\\ Endenicher Allee 60\\ 53115 Bonn, Germany\\
Present address: Department of Mathematics\\ Courant Institute of Mathematical Sciences \\New York University, 251 Mercer St.\\
New York, NY 10012-1110, USA\\
   \printead{e2}}
\end{aug}

\runauthor{A. Bovier and L. Hartung}

\begin{abstract}
 We extend the results of Arguin et al \cite{ABK_E} and A\"\i{}d\'ekon et al \cite{ABBS} on the 
 convergence of the extremal process of branching Brownian motion by adding an extra dimension 
 that encodes the "location" of the particle in the underlying Galton-Watson tree. We show that the
 limit is a cluster point process on $\R_+\times \R$ where each cluster is the atom of a 
 Poisson point process on $\R_+\times \R$ with a random intensity measure $Z(dz) \times 
 C\eee^{-\sqrt 2x}dx$, where the random measure is explicitly constructed from  the derivative 
 martingale. This work is motivated by an analogous result for the Gaussian free field 
 by Biskup and Louidor \cite{biskuplouidor16}.
\end{abstract}

\begin{keyword}[class=MSC]
\kwd[Primary ]{60J80}
\kwd{60G70}
\kwd[; secondary ]{ 82B44}
\end{keyword}

\begin{keyword}
\kwd{Gaussian processes, branching Brownian motion}
\kwd{extremal processes,  cluster processes, multiplicative chaos}
\end{keyword}

\end{frontmatter}


\section{Introduction} 

Over the last years, the analysis of the extremal process of so-called 
\emph{log-correlated}
processes has been studied intensively. One prime example was the construction of the 
extremal process of branching Brownian motion \cite{ABK_E,ABBS} and branching 
random walks
\cite{Madaule11}. For recent reviews see, e.g. \cite{bovier16,shi16}. 
The processes appearing here, Poisson point processes with random 
intensity (Cox processes, see \cite{Cox55})
decorated by a cluster process representing clusters of particles that have rather recent 
common ancestors, are widely believed to be
universal for a wide class of log-correlated processes. In particular, it is expected for the 
discrete Gaussian 
free field, and  results in this direction have been proven by  Bramson, Ding, and 
Zeitouni
 \cite{BraDiZei} and Biskup and Louidor \cite{BisLou13,biskuplouidor16}. These results 
 describe the statistics of the positions  ($=$ values)  of the extremal points of these processes. In 
 extreme value 
 theory (see e.g. \cite{LLR}) it is customary to give an even more complete description of 
 extremal 
 processes that also encode the \emph{locations of the extreme points} (``complete 
 Poisson convergence"). In the case of the two-dimensional 
 Gaussian free field, Biskup and Louidor \cite{BisLou13} 
 conjectured and recently proved \cite{biskuplouidor16}  the following  result.
 For $(i,j)\in (1,\dots, n)^2$, let $X^n$  be the centred Gaussian process indexed by
$ (1,\dots, n)^2$ with covariance\footnote{We change the normalisation of the variance so 
that the results compare better to BBM.}
\be\Eq(gff.1)
\E \left(X_{(i,j)}^nX_{(k,l)}^n\right) = \pi G^n((i,j),(k,l)),
\ee
where $G^n$ is the Green function of simple random walk on $(1,\dots, n)^2$, killed upon exiting
this domain. It is now proven that,
with $m_n(u)\equiv \sqrt {2} \ln n^2-\frac 3{2\sqrt 2} \ln\ln  n^2$,
the family of point  processes on $\R$
\be\Eq(gff.2)
\sum_{1\leq i,j\leq n} \d_{X_{(i,j)}-m_n}
\ee
converges to a process of the form 
\be\Eq(gff.3)
\sum_{i\in \N}\sum_{j\in \N} \d_{p_i+\D^{(i)}_j},
\ee
where the $p_i$ are the atoms of a Poisson point process with random intensity measure
$Z\eee^{-\sqrt {2}u}du$,
for a random variable $Z$, and $\D_j^{(i)}$ are the atoms of iid copies $\D^{(i)}$ of a certain 
point process $\D$ on $[0,-\infty)$.
The extended version of this result reads as follows.
Define the point processes,
\be\Eq(gff.4)
\PP_n\equiv \sum_{1\leq i,j\leq n}  \d_{(i/n,j/n),X_{(i,j)}-m_n},
\ee
on $(0,1]^2\times \R$. Then, $\PP_n$ converges to a point process $\PP$ on  $(0,1]^2\times \R$
of the form 
\be\Eq(gff.5)
\sum_{i\in \N}\sum_{j\in \N} \d_{x_i,p_i+\D^{(i)}_j},
\ee
where $(x_i,p_i)$ are the atoms of a Poisson point process on  $(0,1]^2\times \R$
with random intensity measure $Z(dx)\times \eee^{-\sqrt {2}u}du$, 
where $Z(dx) $ is some random measure on $(0,1]^2$.  
 Biskup and Louidor first  {proved}  in  \cite{BisLou13}  a slightly weaker result for the point process 
of \emph{local extremes:} Let $r_n$ be a sequence such that $r_n\uparrow \infty$ and $r_n/n
\downarrow 0$, and define 
\be\Eq(gff.6)
\eta_n\equiv \sum_{1\leq i,j\leq n}  \d_{\left((i/n,j/n),X_{(i,j)}-m_n\right)}
\1_{\left\{X_{(i,j)}= \max \left(X_{(k,\ell)}:  |k-i|<r_n, |\ell-j|<r_n \right)\right\}}.
\ee
Then $\eta_n$ converges to the  Poisson point process on  $(0,1]^2\times \R$
with random intensity measure $Z(dx)\times \eee^{-\sqrt {2}u}du$,

The purpose of this article is to prove the analog of the full result  for branching Brownian 
motion.  To do so, we need to decide on what should replace the square $(0,1]^2$ in 
that case. 
Before we do this, let us briefly recall the construction of branching Brownian motion. We 
start with 
a continuous time Galton-Watson process \cite{AN} with branching mechanism $p_k, k\geq 1$, normalised such that $\sum_{i=1}^\infty p_k=1$,
$\sum_{k=1}^\infty k p_k=2$ and $K=\sum_{k=1}^\infty k(k-1)p_k<\infty$. At any time $t$
 we may label the endpoints of the process $i_1(t),\dots, i_{n(t)}(t)$, where $n(t)$ is the number of 
 branches at time $t$. Note that, with this choice of normalisation, we have that 
 $\E  n(t)=\eee^t$. 
 Branching Brownian motion is then constructed by starting a Brownian motion at the 
 origin at time 
 zero, running it until the first time the GW process branches, and then starting independent Brownian motions for
 each branch of the GW process starting at the position of the original BM at the branching time. Each of these runs again 
 until the next branching time of the GW occurs, and so on.

 We denote the positions of the $n(t)$ particles at time $t$ by $x_1(t),\dots, x_{n(t)}(t)$. 
 Note that, of course, the positions of these particles do not reflect the position of the 
 particles ``in the tree". 

We now want to embed the leaves of a Galton-Watson process 
 into some finite dimensional space (we  choose $\R_+$) in a 
consistent way that respects 
the
natural tree distance. Since we already know from \cite{ABK_G} that the 
(normalised) genealogical distance of
extreme particles is asymptotically either zero or one, one should expect that the 
resulting process should again be Poisson in this space. In the case of deterministic 
binary  branching 
at integer times,  the leaves of the tree at time $n$ are naturally   labelled by sequences
 $\s^n\equiv (\s_1\s_2\dots\s_n)$, with $\s_\ell\in \{0,1\}$. These sequences can  be 
 naturally mapped 
 into $[0,1]$ via 
 \be\Eq(gff.6.1)
 \s^n\mapsto \sum_{\ell=1}^n \s_\ell 2^{-\ell-1}\in [0,1].
 \ee
Moreover, the limit, as $n\uparrow\infty$ of the image of this map is $[0,1]$.
 In the next section we construct an analogous map for the Galton-Watson process.

The remainder of this paper is organised as follows. In Section 2 we construct an embedding of the Galton-Watson tree 
into $\R_+$ that allows to locate particles "in the tree". In Section 3 we state our main results on the convergence of 
the two-dimensional extremal process of BBM. In Section 4 we analyse the geometric properties of the embedding constructed in Section 2. In Section 5 we recall the $q$-thinning from Arguin et al. \cite{ABK_P}. In Section 6 we give the 
proofs of the main convergence results announced in Section 3.

\noindent\textbf{Acknowledgements.} We thank an anonymous referee for a very careful reading of our paper and for numerous valuable suggestions.

\section{The embedding}

Our goal is to define a map $\g:\{1,\dots,n(t)\}\to \R_+$ in such a way that it encodes the 
genealogical structure of the underlying 
supercritical Galton-Watson process. 

 Let us define the set of (infinite) multi-indices
\be\Eq(multi.1)
\mathbf{I}\equiv \Z_+^\N,
\ee
and let $\mathbf{F}\subset \mathbf{I}$ denote the subset of multi-indices that contain 
only a finitely many entries that are different from zero. Ignoring leading zeros, we see
that 
\be\Eq(multi.2)
\mathbf{F} = \cup_{k=0}^\infty \Z_+^k,
\ee
where $\Z_+^0$ is either the empty multi-index or the multi-index containing only 
zeros.

A continuous-time Galton-Watson process will be encoded by the set of branching times,
$\{t_1<t_2<\dots< t_{W(t)}<\dots\}$  (where $W(t)$ denotes the number of branching 
times up to time $t$) and by a consistently assigned set of multi-indices
for all times $t\geq 0$. 
To do so, we construct, for a given tree, 
the sets of multi-indices, $\t(t)$ at time $t$ as follows.


\begin{figure}[htbp]
 \includegraphics[width=15cm]{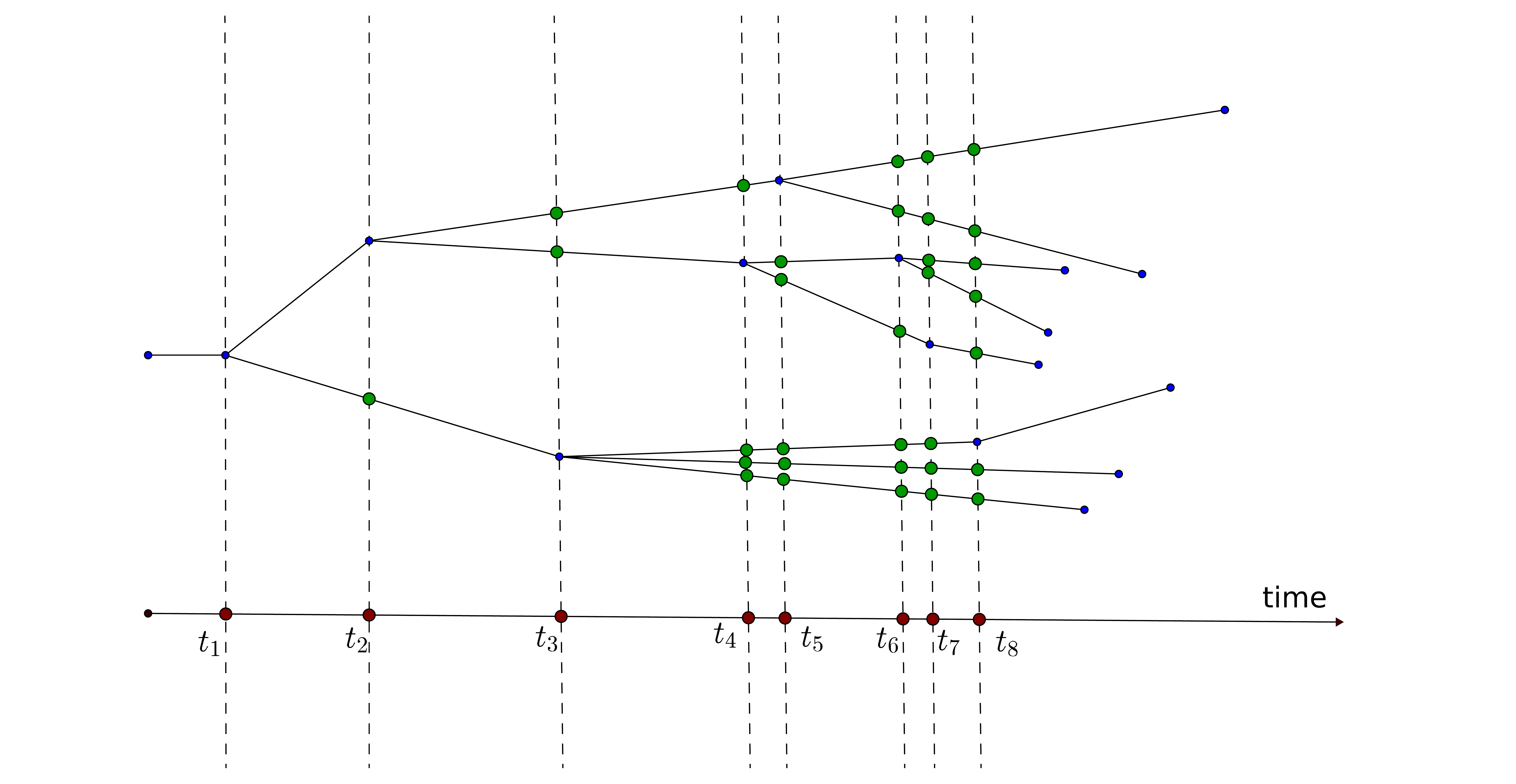}
 \caption{Construction of $\wt T$: The green nodes were introduced into the tree `by hand'.}\label{figure.1}
\end{figure}

\begin{itemize}
 \item $\{(0,0,\dots)\} =\{u(0)\}=\t(0)$.
 \item  for all $j\geq 0$,  for all  $t\in [t_j,t_{j+1})$,   $\t(t)=\t(t_j)$.
 \item If $u\in \t(t_j)$ then $u+(\underbrace{0,\dots,0}_{W(t_j) \times 0}, k,0,\dots)\in 
 \t(t_{j+1})$ if $0\leq k\leq l^u(t_{j+1})-1$, where
 \be\Eq(map.1)
l^u(t_j)=\# \{\mbox{ offsprings of  the particle corresponding to }u\, \mbox{at time}\, t_j \}.
\ee
\end{itemize}
Note that we use the convention that, if a given branch of the tree does not "branch" at time $t_j$, we
add to the underlying 
Galton-Watson  at this time an extra vertex where $l^u(t_j)=1$.
 (see Figure \ref{figure.1}. The new vertices are 
the thick dots).  We call the
resulting tree $\wt T_t$.


We can relate the assignment of labels in a backwards consistent fashion as follows.
For  $u\equiv (u_1,u_2,u_3,\dots)\in \Z_+^\N$, we define the function $u(r), r\in \R_+$,  through
\be\Eq(multi.3)
u_\ell(r)\equiv \begin{cases}  u_\ell,&\,\, \hbox{\rm if}\,\, t_\ell\leq r,\\
0,&\,\, \hbox{\rm if}\,\, t_\ell> r.
\end{cases}
\ee
Clearly,  if $u(t)\in \t(t)$ and  $r\leq t$, then $u(r)\in \t(r)$.  This allows to define the 
\emph{boundary} of the 
tree at infinity  as follows:
\be\Eq(multi.4)
\del \mathbf{T} \equiv \left\{ u\in \mathbf{I}: \forall t<\infty, u(t)\in \t(t)\right\}.
\ee
\index{boundary!of Galton-Watson process}
Note that $\del\mathbf{T}$ is an \emph{ultrametric} space equipped with the 
ultrametric $m(u, v)\equiv \eee^{-d(u,v)}$, where  
$d(u,v)=\sup \{t\geq 0: u(t)=v(t)\}$ is the time of 
 their most recent common ancestor.

%
%
 In this way each 
 leave of the Galton-Watson tree at time $t$,  $i_{k}(t)$ with $k\in\{1,\dots,n(t)\}$ is identified with some multi-label $u^k(t)\in \t(t)$. Then define
\be\Eq(map.3)
\gamma(u(t))\equiv\sum_{j=1}^{W(t)} {u_{j}(t)}\eee^{-t_j}.
\ee
For a given $u$, the function $(\g(u(t)), t\in \R_+)$ describes a trajectory of a particle in $\R_+$. 
The important point is that, for a  fixed particle, this trajectory converges to some point $\g(u)\in 
\R_+$, as $t\uparrow \infty$, almost surely. Hence also the sets $\g(\t(t))$ converge, for any realisation of the tree, to some (random) set $\g(\t(\infty))$.  

\begin{remark}
The labelling of the GW-tree is a slight variant of the familiar Ulam-Neveu-Harris labelling 
(see e.g. \cite{HaHa06}). In our labelling the added zeros keep track of the order in which 
branching occurred in continuous time.
We believe that this or an equivalent construction must be standard, but we have not been able to 
find it for continuous time trees in the literature.
\end{remark}

In addition, in branching Brownian motion, there is also the position of the Brownian motion 
$x_k(t)$ of the $k$-th particle at time $t$. Hoping that there will not be too much confusion, 
we will often write $\g(x_k(t))\equiv \g(u^k(t))$. Thus to any "particle" at time $t$ we can now associate the position on $\R\times\R_+$, $(x_k(t), \g(u^k(t)))$.

\section{The extended convergence result}
In this section we state the analog to \eqv(gff.5) for branching Brownian motion. First let us recall the limit of the extremal process. Bramson \cite{B_M} and 
Lalley and Sellke \cite{LS} show that, with $m(t)=\sqrt 2t -\frac{3}{2\sqrt 2} \ln t$,
\be\Eq(extremal.1.1)
\lim_{t\uparrow\infty}\P\left(\max_{k\leq n(t)}x_k(t)-m(t)\leq x\right)=\omega(x)=\E 
\left[\eee^{-CZ\eee^{-\sqrt 2 x}}\right],
\ee
for some constant $C$, and where $Z\equiv \lim_{t\uparrow\infty} Z_t$ is the limit of the derivative martingale
\be\Eq(Z.0)
Z_t\equiv \sum_{j\leq n(t)}(\sqrt 2 t-x_j(t))\eee^{\sqrt 2(x_j(t)-\sqrt 2 t)}.
\ee
In \cite{ABK_E} and \cite{ABBS} it was shown that the  process,
\be\Eq(extremal.1)
\EE_t\equiv \sum_{k=1}^{n(t)}\d_{x_k(t)-m(t)}
\ee
converges, as $t\uparrow \infty$, in  law to the process
 \be\Eq(extremal.2)
 \EE =\sum_{k,j}\d_{\eta_k+\D^{(k)}_j},
\ee
where $\eta_k$ is the $k$-th atom of a Cox process with random intensity 
measure $CZ \eee^{-\sqrt 2 y}dy$. The $\D^{(k)}_i$ are the atoms of independent and identically distributed  point processes $\D^{(k)}$, which are 
 copies of the limiting process 
\be\Eq(extremal.3)
\D \stackrel{D}{=} \lim_{t\uparrow \infty}
\sum_{i=1}^{n(t)}\d_{\tilde x_i(t)-\max_{j\leq n(t)}\tilde x_j(t)},
\ee 
where $\tilde x(t)$ is a BBM conditioned on $\max_{j\leq n(t)} \tilde x_j(t)\geq \sqrt 2 t$.

Using the embedding $\g$ defined in the previous section, we  now state the following 
 theorem, that exhibits more precisely the nature of the Poisson points and the 
 genealogical structure of the extremal particles.

\begin{theorem}\TH(thm.extend)
The point process $
\wt\EE_t\equiv \sum_{k=1}^{n(t)} \d_{(\g(u^k(t)), x_k(t)-m(t))}\rightarrow \wt\EE$ on $\R_+\times \R$, as
$t\uparrow \infty$,
where 
\be
\wt\EE\equiv \sum_{i,j}\d_{(q_i,p_i)+(0,\D^{(i)}_j)},
\ee
where  $(q_i,p_i)_{i\in \N}$ are the atoms of a Cox process on $\R_+\times \R$ with intensity measure
$Z( dv)\times C\eee^{-\sqrt{2} x}dx$, where $Z(dv)$ is a random measure on $\R_+$, characterised in Proposition 
\thv(lem.Z), and $\D_j^{(i)}$  are the atoms of  independent and identically distributed point processes $\D^{(i)}$ as in \eqv(extremal.2) .
\end{theorem} 

\begin{remark}
The nice feature of the process $\wt\EE_t$  is that it allows to 
visualise the different clusters $\D^{(i)}$ corresponding to the different point of the 
Poisson process of cluster extremes. In the process 
$\sum_{k=1}^{n(t)} \d_{  x_k(t)-m(t)}$ considered in earlier work, all these points get superimposed and cannot be disentangled. In other words, the process $\wt\EE$ encodes both the values and the (rough) genealogical structure of the extremes of BBM.
\end{remark}

The measure $Z(dv)$ in an interesting object in itself.
For $v,r\in\R_+$ and $t>r$, we define
\be\Eq(Z.1)
Z(v,r,t)=\sum_{j\leq n(t)}(\sqrt 2 t-x_j(t))\eee^{\sqrt 2(x_j(t)-\sqrt 2 t)}\1_{\g(x_j(r))\leq v},
\ee
which is a truncated version of the usual derivative martingale $Z_t$. In particular, observe that $Z(\infty,r,t)=Z_t$.

\begin{proposition}\TH(lem.Z)
 For each $v\in\R_{+}$ the limit $\lim_{r\uparrow \infty}\lim_{t\uparrow \infty}Z(v,r,t)$  exists almost surely. Set  
 \be\Eq(Z.2)
 Z(v)\equiv\lim_{r\uparrow \infty}\lim_{t\uparrow \infty}Z(v,r,t) .
 \ee
 Then $0\leq Z(v) \leq Z$, where $Z$ is the limit of the derivative martingale. Moreover, $Z(v)$ is monotone increasing in $v$ and the corresponding measure $Z(dv)$ is a.s. non-atomic.
\end{proposition}

The measure $Z(v)$ is the analogue of the corresponding "derivative martingale 
measure" studied in Duplantier et al \cite{Dup14a,Dup14b} and Biskup and Louidor
\cite{BisLou13, BisLou14} in the context of the 
Gaussian free field and in \cite{BRV,BKNSW} for the critical Mandelbrot multiplicative cascade. For a review, see 
Rhodes and Vargas \cite{RV14}. The objects are examples of what is known as 
\emph{multiplicative chaos} that was introduced by Kahane \cite{Kahane85}.

\section{Properties of the embedding}

We need the  three basic properties of $\g$. Lemma \thv(lem.prop.1) 
states that the map $\g(x_k(t))$ converges for all extremal particles, as 
$t\uparrow \infty$, and is well approximated by the information on the tree up to a fixed time $r$. 
\begin{lemma}\TH(lem.prop.1)
 Let $D\subset \R$ be a compact set. Define, for $0\leq r<t<\infty$,  the events
 \be\Eq(map.4.1)
 \mathcal{A}_{r,t}^{\g}(D)=\left\{\forall k \mbox{ with } x_k(t)-m(t)\in D \mbox{ : }\g(x_k(t))-\g(x_k(r))\leq \eee^{-r/2}\right\}.
 \ee
 For any $\e>0$ there exists $0\leq r(D,\e)<\infty$ such that, for any $r>r(D,\e)$ and
 $t>3r$
 \be\Eq(map.4)
 \P\left(\left(\mathcal{A}_{r,t}^{\g}(D)\right)^c\right)<\e. 
 \ee
\end{lemma}
\begin{proof} Set $\overline D\equiv \sup\{x\in D\}$ and $\underline D\equiv \inf\{x\in D\}$. 
Let $\e>0$. Then, 
  by Theorem 2.3 of \cite{ABK_G}, for each $\e>0$ there exists  $r_1<\infty$  such that,   for all $t>3 r_1$,
  \bea\Eq(map.5)
  \P\left(\left(\mathcal{A}_{r,t}^{\g}(D)\right)^c\right)
& \leq &\P\big(\exists k: 
  x_k(t)-m(t)\in D,
  \forall_{s\in[r_1,t-r_1]}: x_k(s)\leq \overline D+ E_{t,\a}(s)\nonumber\\
  &&\qquad\mbox{ but }
  \g(x_k(t))-\g(x_k(r))>\eee^{-r/2}\big)
 +\e/2,
  \eea
  where $0<\a<\frac{1}{2}$ and $E_{t,\a}(s)=\frac{s}{t}m(t)- f_{t,\a}(s)$ and $f_{t,\a}=(s\wedge (t-s))^{\a}$. Using the "many-to-one lemma" (see Theorem 8.5 of \cite{Hardy})), the probability in \eqv(map.5) is bounded from above by 
  \be\Eq(map.6)
  \eee^t\P\left(x(t)\in m(t)+D,\forall_{s\in[r_1,t-r_1]}: x(s)\leq \overline D +E_{t,\a}(s) \mbox{ but } \begin{textstyle}\sum_{j}\end{textstyle}m_j \eee^{-\tilde t_j}\1_{\tilde t_j\in[r,t]}> \eee^{-r/2}\right),
  \ee
   where $x$ is a standard Brownian motion and $(\tilde t_j, j\in \N)$ are the points of a size-biased 
   Poisson point process with intensity measure $2dx$ independent of $x$,
    $m_j$ are independent  random variables uniformly distributed on  
    $\{0,\dots,\tilde l_j-1\}$, where finally 
    $\tilde l_j$ are i.i.d. 
   according to the size-biased offspring distribution, $\P(\tilde l_j=k)=\frac{kp_k}{2}$.
    Due to independence, and since  $m_j\leq \tilde 
   l_j$, the expression  \eqv(map.6) is bounded from above by
   \bea \Eq(map.7)
   &&\eee^t\P\left(x(t)\in m(t)+D,\forall_{s\in[r_1,t-r_1]}: x(s)\leq \overline D+E_{t,\a}(s)\right)
  \nonumber\\&&\quad\times
 \P\left(\begin{textstyle}\sum_{j}\end{textstyle}(\tilde l_j-1) \eee^{-\tilde t_j}\1_{\tilde t_j\in[r,t]}>\eee^{-r/2}\right).
   \eea
   The first probability in \eqv(map.7) is bounded by
   \be\Eq(map.8)
   \P\left(x(t)\in m(t)+D,\forall_{s\in[r_1,t-r_1]}: x(s)-\frac{s}{t}x(t)\leq \overline D-
   \underline D-f_{t,\a}(s)\right).
   \ee
   Using that $\xi(s)\equiv x(s)-\frac{s}{t}x(t)$ is a Brownian bridge from $0$ to $0$ in time $t$ that is independent of $x(t)$,  \eqv(map.8) equals
   \bea\Eq(map.9)
   &&\P\left(x(t)\in m(t)+D\right)\P\left(\forall_{s\in[r_1,t-r_1]}: \xi(s)\leq \overline D-\underline
    D-f_{t,\a}(s)\right)\nonumber\\
   &&\leq \P\left(x(t)\in m(t)+D\right)\P\left(\forall_{s\in[r_1,t-r_1]}: \xi(s)\leq \overline D-\underline D\right) .
   \eea
   Using now Lemma 3.4 of \cite{ABK_G} to bound the last factor of \eqv(map.9) we obtain that \eqv(map.9) is bounded from above by
   \be\Eq(map.10)
   \k \frac{r_1}{t-2r_1}\P\left(x(t)\in m(t)+D\right),
   \ee
   where $\k<\infty$ is a positive constant.
  Using this as an upper bound for the first probability in \eqv(map.7) we can bound \eqv(map.7)  from above by
   \be\Eq(map.11)
   \eee^t\k \frac{r_1}{t-2r_1}\P\left(x(t)\in m(t)+D\right)\P\left(\begin{textstyle}\sum_{j}\end{textstyle}(\tilde l_j-1) \eee^{-\tilde t_j}\1_{\tilde t_j\in[r,t]}>\eee^{-r/2}\right).
   \ee
   By (5.25) of \cite{ABK_G}(or an easy Gaussian computation) this is bounded from above by
    \be\Eq(map.11')
   C\k \frac{r_1t}{t-2r_1}\P\left(\begin{textstyle}\sum_{j}\end{textstyle}(\tilde l_j-1) \eee^{-\tilde t_j}\1_{\tilde t_j\in[r,t]}>\eee^{-r/2}\right),
   \ee
   for some positive constant $C<\infty$. Using the  Markov inequality, \eqv(map.11') is bounded from above by
   \be\Eq(map.12)
   C\k \frac{tr_1}{t-2r_1}\eee^{r/2}\E\left(\begin{textstyle}\sum_{j}\end{textstyle}(\tilde l_j-1) \eee^{-\tilde t_j}\1_{\tilde t_j\in[r,t]}\right),
   \ee
We condition on the $\s$-algebra $\mathcal{F}$ generated by the Poisson points.  Using that $\tilde l_j$ is independent of the 
  Poisson point process $(\tilde t_j)_j$ and  $\sum_{j} \eee^{-\tilde t_j}\1_{\tilde t_j\in[r,t]}$ is measurable with respect to $\mathcal F$ we
   obtain that \eqv(map.12)  is equal to
    \bea\Eq(map.13)
    &&C\k \frac{tr_1}{t-2r_1}\eee^{r/2}\E\left(\E\left(\begin{textstyle}\sum_{j}\end{textstyle}(\tilde l_j-1) \eee^{-\tilde t_j}\1_{\tilde t_j\in[r,t]}\vert\mathcal{F}\right)\right)
    \\\nonumber&&=
   C\k \frac{tr_1}{t-2r_1}\eee^{r/2}\E\left(\begin{textstyle}\sum_{j}\end{textstyle} \eee^{-\tilde t_j}\1_{\tilde t_j\in[r,t]}\E\left((\tilde l_j-1)\vert\mathcal{F}\right)\right).
   \eea
   Since $\E(l_j-1)=\sum_{k}\frac{1}{2}(k-1)kp_k=K/2<\infty$ we have that \eqv(map.13) is equal to
    \be\Eq(map.14)
   C\k K/2 \frac{tr_1}{t-2r_1}\eee^{r/2}\E\left(\begin{textstyle}\sum_{j}\end{textstyle} \eee^{-\tilde t_j}\1_{\tilde t_j\in[r,t]}\right).
   \ee
   By Campbell's theorem (see e.g \cite{Kingman93} ), \eqv(map.14) is equal to
   \be\Eq(map.14')
    C\k K/2 \frac{tr_1}{t-2r_1}\eee^{r/2}\int_r^t \eee^{-x}2dx\leq C\k K\frac{tr_1}{t-2r_1}\eee^{-r/2},
   \ee
   which is smaller than $\e/2$, for all $r$ sufficiently large and $t>3r$.
\end{proof}

The second lemma now ensures that $\g$ maps particles, that are extremal, with low probability to a very small neighbourhood of a fixed $a\in\R$.
\begin{lemma}\TH(lem.prop.2)
 Let $a\in\R_+$ and $D\subset \R$ be a compact set.  Define the event
 \be\Eq(map.15.1)
 \mathcal{B}^{\g}_{r,t}(D,a,\d)=\left\{\forall k \mbox{ with } x_k(t)-m(t)\in D \mbox{: }\g(x_k(r))\not\in[a-\d,a]\right\}.
 \ee
 For any $\e>0$ there exists $\d>0$ and  $r(a,D,\d,\e)$ such that, for any $r>r(a,D,\d,\e)$ and $t>3r$
 \be\Eq(map.15)
 \P\left(\left(\mathcal{B}^{\g}_{r,t}(D,a,\d)\right)^c\right)<\e.
 \ee
\end{lemma}
\begin{proof} 
 Following the proof of Lemma \thv(lem.prop.1) step by step we arrive
 at the bound 
  \be\Eq(map.16)
  \P\left(\left(\mathcal{B}^{\g}_{r,t}(D,a,\d)\right)^c\right)\leq
    C\k \frac{tr_1}{t-2r_1}\P\left(\begin{textstyle}\sum_{j}\end{textstyle} m_j \eee^{-\tilde t_j}\1_{\tilde t_j\in[0,r]}\in[a-\d ,a]\right).
  \ee

 We rewrite the probability in \eqv(map.16) in the form
 \be\Eq(map.20)
 \sum_{i^{*}=1}^\infty \P\left(i^{*}=\inf\{i:m_i\neq 0\},\begin{textstyle}\sum_{j\geq i^{*}}
 \end{textstyle}m_j \eee^{-\tilde t_j}\1_{\tilde t_j\in[0,r]}\in[a-\d ,a]\right).
 \ee
 Consider first $\P\left(i^{*}=\inf\{i:m_i\neq 0\}\right)$. This probability is equal to 
 \be\Eq(map.28)
 \P\left(\forall_{i\leq i^{*}}:m_i=0 \mbox{ and } m_{i^{*}}\neq 0\right)
 =\E\left[\left(1-\frac{1}{l_{i^{*}}}\right)\prod_{j=1}^{i^{*}-1}\frac{1}{l_j}\right].
 \ee
Using that  the $l_j$ are iid together with the simple bound $\E\left(l_j^{-1}\right)\leq 
 \frac{1+p_1}{2}$, we see that  \eqv(map.28) is bounded from above by
 \be\Eq(map.29)
 \left(\frac{1+p_1}{2} \right)^{i^{*}-1}.
 \ee
 Since $\frac{1+p_1}{2}<1$ by assumption on $p_1$ we can choose, for each $\e'>0$  $K(\e')<\infty$ such that
 \be\Eq(map.30)
 \sum_{i^{*}=K(\e')+1}^{\infty}\left(\frac{1+p_1}{2}\right)^{i^{*}-1} <\e' .
\ee
Hence we bound \eqv(map.20) by 
\be\Eq(map.neu.2)
\sum_{i^{*}=1}^{K(\e')}\P\left(i^{*}=\inf\{i:m_i\neq 0\},\begin{textstyle}\sum_{j\geq i^{*}}
 \end{textstyle}m_j \eee^{-\tilde t_j}\1_{\tilde t_j\in[0,r]}\in[a-\d ,a]\right) +\e'.
\ee
 We rewrite
 \be\Eq(map.neu.1)
 \begin{textstyle}\sum_{j\geq i^{*}}
 \end{textstyle}m_j \eee^{-\tilde t_j}\1_{\tilde t_j\in[0,r]}= m_{i^*}\eee^{-\tilde t_{i^{*}}}
 \1_{\tilde t_{i^{*}}\in[0,r]} \left(1+m_{i^*}^{-1}\sum_{j> i^*}m_j \eee^{-(\tilde t_j-\tilde t_{i^{*}})}
 \1_{\tilde t_j-t_i^*\in[0,r-t_{i^*}]} \right)
 \ee
 Next, we estimate the probability that $\tilde t_{i^*}$ is large.
 Observe that $\tilde t_{i^*}=\sum_{i=1}^{i^*}s_i$ where $s_i$ are iid exponentially distributed 
 random variables with parameter $2$.  This implies that  
  $\tilde t_{i^{*}} $ is Erlang$(2,i^{*})$. Thus 
 \be\Eq(map.neu.3)
 \P\left(\tilde t_{i^*}>r^\a\right)=\eee^{-2r^\a}\begin{textstyle}\sum_{i=0}^{i^{*}}
\end{textstyle}
\frac{(2r^{\a})^{i}}{i!}\leq 
\eee  (2r^\a)^{K(\e')}\eee^{-2r^\a},
\mbox{ for all }i^{*}\leq K(\e') .
 \ee
 Next we want to replace $\tilde t_{i^*}$ in the indicator function in  \eqv(map.neu.1)
 by a non-random quantity  $r^\a$, for some $0<\a<1$, in order to have a bound that depends only on 
 the differences $\tilde t_j-\tilde t_{i^*}$. 
 Note first that 
 \bea\Eq(summe.1)
 &&\sum_{j> i^*}m_j \eee^{-(\tilde t_j-\tilde t_{i^{*}})}
 \1_{\tilde t_j-t_i^*\in[0,r-t_{i^*}]} -
 \sum_{j> i^*}m_j \eee^{-(\tilde t_j-\tilde t_{i^{*}})}
 \1_{\tilde t_j-t_i^*\in[0,r-r^\a]}
 \\\nonumber
 &&=\sum_{j> i^*}m_j \eee^{-(\tilde t_j-\tilde t_{i^{*}})}
 \1_{\tilde t_j-t_i^*\in[r-r^\a,r-t_{i^*}]}
 \leq 
 \sum_{j> i^*}m_j \eee^{-(\tilde t_j-\tilde t_{i^{*}})}
 \1_{\tilde t_j-t_i^*\in[r-r^\a,r]}.
 \eea
 Using the fact that  $m_j\leq \tilde l_j-1$, for all $j$
 and the  Markov inequality, we get that 
 \bea \Eq(map.neu.5)
 &&\P\left(\begin{textstyle}\sum_{j> i^*}\end{textstyle}m_j \eee^{-(\tilde t_j-\tilde  t_{i^*})}\1_{\tilde t_j-\tilde t_{i^*}\in[r-r^\a,r ]}>\eee^{-r/2}\right)\nonumber\\
 &&\leq  \eee^{r/2}\E\left(\begin{textstyle}\sum_{j> i^*}\end{textstyle}(\tilde l_j-1) 
 \eee^{-(\tilde t_j-\tilde  t_{i^*})}\1_{\tilde t_j-\tilde t_{i^*}\in[r-r^\a,r ]}\right).
 \eea
 Using Campbell's theorem as in \eqv(map.13), we see  that 
 the second line in \eqv(map.neu.5) is equal to
 \be\Eq(map.nue.5.1)
 \eee^{r/2} K/2 \int_{r-r^\a}^r\eee^{-x}2dx=K\left(\eee^{-r/2+r^\a}-\eee^{-r/2}\right).
 \ee
 For any $\e'>0$, there exists $r_0<\infty$, such that, for all $r>r_0$, the 
  probabilities in \eqv(map.neu.3) and \eqv(map.neu.5)  are smaller than $\e'$.
   On the the event 
   \be\Eq(event.1)
   \mathcal{D}=\{t_{i^*}\leq r^\a \}\cap\left\{\sum_{j> i^*}m_j \eee^{-(\tilde t_j-\tilde  t_{i^*})}
   \1_{\tilde t_j-t_i^*\in[r-r^\a,r ]}\leq \eee^{-r/2}\right\},
   \ee
  which has probability at least $1-2\e'$,  we can bound \eqv(map.neu.2) in a nice way. Namely,
since $m_{i^{*}}\geq 1$ by definition and  $m_j$ are chosen uniformly from 
$(0,\dots,l_{j}-1)$ and independent of  $\{t_j\}_{j\geq 1}$.   Moreover,
 $\sum_{j> i^{*}}m_j \eee^{-(\tilde t_j-\tilde t_{i^{*}})}\1_{\tilde t_j-t_{i^{*}}\in[0,r-r^\a]}\geq 0$ 
  is also independent of $t_{i*}$. It follows that  \eqv(map.neu.2) is bounded from 
   above by 
 \be\Eq(map.23)
\sum_{i^{*}=1}^{K(\e')}
\P\left(i^{*}=\inf\{i:m_i\neq 0\}\right)\max_{b\in[0,1]}\P\left(\{\eee^{-\tilde 
 t_{i^{*}}} \in[b-\d-\eee^{-r/2} ,b] \}\land \{t_{i^*}\leq r^\a\}\right)+3\e'.
 \ee
 Using  the bound on the first probability in \eqv(map.23) given  in \eqv(map.29), 
 one sees that \eqv(map.23) is bounded from above by
 \be \Eq(map.25)
\sum_{i^{*}=1}^{K(\e')}
\left(\frac{1+p_1}{2} \right)^{i^{*}-1}
\max_{b\in[\d+\eee^{-r^\a}+\eee^{-r/2},1]}\P\left( t_{i^{*}} \in\left[-\log b ,-\log\left(b-\d -\eee^{-r/2}\right)\right]   \right) +3\e'
 \ee 
 Recalling  that $t_{i^{*}} $ is Erlang$(2,i^{*})$ distributed,  we have that
 \bea\Eq(map.27)
 && \P\left( t_{i^{*}}  \in\left[-\log b ,-\log\left(b-\d-\eee^{-r/2}\right)\right]   \right)\nonumber\\
  &&= 
     \sum_{i=0}^{i^{*}-1}\frac 1{i!}\left(f_i(b-\d-\eee^{r/2})-f_i(b)\right),
 \eea
 where we have set $f_i(x)= x^2\left(-2\log(x)\right)^i$. 
 By the mean value theorem, uniformly on $b\in [\d+\eee^{-r^\a}+\eee^{-r/2},1]$, 
 \be\Eq(mean-value.1)
0\leq  f_i(b)-f_i(b-\d-\eee^{-r/2})\leq 2(2{r^\a} ))^{i}(i +2r^\a)\left(\d+\eee^{-r/2}\right).
 \ee
 Inserting this bound into \eqv(map.27), we get that, for $i^*\leq K(\e')$, 
  \bea
 \Eq(mean-value.2)
 &&\max_{b\in[\d+\eee^{-r^\a}+\eee^{-r/2},1]}\P\left( t_{i^{*}} \in\left[-\log b ,-\log\left(b-\d -\eee^{-r/2}\right)\right]   \right)  \\\nonumber
 &&\leq 
  4(\d+\eee^{-r/2}) 
\sum_{i=0}^{i^{*}}
\frac{1}{i!} (2r^\a)^{i}
\nonumber\Eq(map.32)
\leq4 \eee  (\d+\eee^{-r/2})  \eee^{2r^\a}.
 \eea
 Now we choose $r$ so big that $4 \eee^{-r/2+2 r^\a+1}\leq \e'/2$ and then $\d$ so small that $\d 4\eee^{2r^\a+1}\leq \e'/2$, so that the entire expression on 
 the right is bounded by $\e'$.
Collecting the bounds in \eqv(map.neu.3), \eqv(map.neu.5) and \eqv(map.32) 
implies \eqv(map.15) if $\e'=\e/4$
 \end{proof}

The following lemma asserts that any two points that get close to the maximum of BBM,
 have distinct images under the map $\g$, unless the time of the most recent common ancestor is
large.  This implies in particular that the positions of the cluster extremes all differ in the second coordinate. This lemma is not strictly needed in the proof
of our main theorem, but we find it nice to make this point explicit. The proof uses largely the same arguments that were used in the proofs of Lemmas \thv(lem.prop.1)
and \thv(lem.prop.2).
 
 \begin{lemma}\TH(lem.prop.3)
 Let $D\subset \R$ be a compact set. For any $\e>0$ there exists $\delta>0$ and $r(\d,\e)$ such that, for any $r>r(\d,\e)$ and $t>3r$ 
 \be\Eq(map.40)
 \P\left(\exists_{ i,j\leq n(t): d(x_i(t),x_j(t))\leq r}: x_i(t),x_j(t)\in m(t)+D,\vert \g(x_i(t))-\g(x_j(t))\vert\leq \d\right)<\e.
 \ee
\end{lemma}

\begin{proof}
To control \eqv(map.40), we first use that, by Theorem 2.1 in \cite{ABK_G}, 
for any $\e'$, there is $r_1<\infty$, such that, for all $t\geq 3r_1$, and $r\leq t/3$,
 the event 
\be\Eq(gene.1)
\{ \exists_{ i,j\leq n(t): d(x_i(t),x_j(t))\in (r_1,r)}, x_i(t),x_j(t)\in m(t)+D\}
\ee
has probability smaller than $\e'$. Therefore, 
\bea\Eq(gene.2)
&&\hspace{-4mm}\P\left(\exists_{ i,j\leq n(t): d(x_i(t),x_j(t))\leq r}: x_i(t),x_j(t)\in m(t)+D,\vert \g(x_i(t))-\g(x_j(t))\vert\leq \d\right)\\\nonumber
&&\hspace{-4mm}\leq 
\P\left(\exists_{ i,j\leq n(t): d(x_i(t),x_j(t))\leq r_1}: x_i(t),x_j(t)\in m(t)+D,\vert \g(x_i(t))-\g(x_j(t))\vert\leq \d\right) +\e'.
\eea
The nice feature of the probability in the last line is that $r_1$ is now independent of $r$.

 To bound the probability in the last line, we proceed as follows: at time $r_1$, there are $n(r_1)$ particles alive. From these we select the ancestors of the particles $i$ and $j$. This gives 
 at most $n(r_1)^2$ choices. The offspring of these particle are then independent, conditional on what happened up to time $r_1$, i.e. the $\s$-algebra $\FF_{r_1}$. 
 We denote the offspring of these two particles 
 starting from time $r_1$ by $\tilde x^{(1)}$ and $\tilde x^{(2)}$.  In this way, we write this  probability in the form
 \be\Eq(seltsam.1)
 \E\left[ \sum_{\ell\neq \ell'=1}^{n(r_1)} \P\left(   \dots \big |\FF_{r_1}\right)\right],
 \ee
 where
 \bea\Eq(seltsam.2)
 &&\hspace{-5mm} \P\left(   \dots \big |\FF_{r_1}\right)\nonumber\\\nonumber
 &&\hspace{-5mm} =\P\Bigl(\exists_{ i\leq n^{(1)}(t-r_1),j\leq n^{(2)}(t-r_1)}:   x_\ell(r_1)+ \tilde x^{(1)}_i(t-r_1), x_{\ell'}(r_1)+ x^{(2)}_j(t-r_1)\in m(t)+D, \\ && \quad \vert \g( x_\ell(r_1)+  x^{(1)}_i(t-r_1))-
 \g(x_{\ell'}(r_1)+x^{(2)}_j(t-r_1))
 \vert\leq \d\Big|\FF_{r_1}\Bigr).
 \eea
 The conditional probability is a function of $x_\ell(r_1)$ and $x_{\ell'}(r_1)$ only, and we will bound it uniformly on a set of large probability. 
 Note first that we can chose as finite enlargement,    $\wt D$, of the set $D$  (depending only on the value of $r_1$), such that  such that $D+x_k(r_1)\subset \wt D$  and $D+x_\ell(r_1)\in \wt D $
  with probability at least $1-\e''$. For such $x_{\ell}(r_1), x_{\ell'}(r_2)$, \eqv(seltsam.2) is bounded from above by
\bea\Eq(seltsam.3)\nonumber
&&  \P\Bigl(\exists_{ i\leq n^{(1)}(t-r_1),j\leq n^{(2)}(t-r_1)}:   \tilde x^{(1)}_i(t-r_1), \tilde x^{(2)}_j(t-r_1)\in m(t)+\wt D, \\ && \quad \vert \g( x_\ell(r_1)+  \tilde x^{(1)}_i(t-r_1))-\g(x_{\ell'}(r_1)+
\tilde x^{(2)}_j(t-r_1))
 \vert\leq \d\Big|\FF_{r_1}\Bigr).
 \eea
 Next, we notice that, at the expense of a further error $\e''$, we can introduce the condition that the paths stay below the curves $E_{t-r_1,\a}(s)$, for all $(r_2,t-r_1-r_2)$, for some $r_2$
 depending only on $\e''$. Using the independence of the BBMs $\tilde x^{(1)}$ and $\tilde x^{(2)}$, and proceeding otherwise  as in \eqv(map.7), we can bound \eqv(seltsam.3) from above by 
%
\bea\Eq(map.70)
&&\hspace{-7mm} \e''+ \left(  C\k \frac{(t-r_1)r_2}{t-r_1-2r_2}\right)^2
\\\nonumber&&\hspace{-7mm}\times\;\P\left(\Big\vert \g(x_\ell(r_1))-\g(x_{\ell'}(r_1))+\begin{textstyle}\sum_{k}\end{textstyle}m_k^j\eee^{-\tilde t_k^j}\1_{\tilde t_k^j\in[r_1,t]}
-\begin{textstyle}\sum_{k'}\end{textstyle}m_{k'}^i\eee^{-\tilde t_{k'}^i}\1_{\tilde t_{k'}^i\in[r_1,t]} 
\Big\vert \leq\d \Big|\FF_{r_1}\right),
  \eea
 where $(\tilde t_k^j,k\in\N)$ and $(\tilde t_{k'}^i,k'\in\N)$ are the points of independent Poisson point processes with intensity $2dx$ restricted to $[r_1,t]$. Moreover, $l_k^j,l_{k'}^i$ are i.i.d.
  according to the size-biased offspring distribution and $m_k^j$ resp. $m_{k'}^i$  are uniformly distributed on $\{0,\dots,l_k^j-1\}$ resp. $\{0,\dots,l_{k'}^i-1\}$. We rewrite \eqv(map.70) as
 \be\Eq(map.42)
 \P\left( \begin{textstyle}
 \sum_{k}\end{textstyle}m_k^j\eee^{-\tilde t_k^j} 
 \1_{\tilde t_k^j\in[r_1,t]} \in\g(x_{\ell'}(r_1))-\g(x_\ell(r_1))+
 \begin{textstyle}\sum_{k'}\end{textstyle}m_{k'}^i\eee^{-\tilde t_{k'}^i}\1_{\tilde t_{k'}^j
 \in[r_1,t]} +[-\d,\d]\Big|\FF_{r_1}\right).
 \ee
As in \eqv(map.20) we rewrite the probability in  \eqv(map.42) as
 \bea\Eq(map.43)
 &&\hspace{-5mm} \sum_{l=1}^\infty\P\Bigl( l=\inf\{k:m_k^j\neq 0\},\;%
\\\nonumber
&&\hspace{-5mm}  \begin{textstyle}\sum_{k\geq l}\end{textstyle}
 m_k^j\eee^{-\tilde t_k^j} \1_{\tilde t_k^j\in[r_1,t]}  \in\g(x_{\ell'}(r_1))-\g(x_\ell(r_1))+\begin{textstyle}
 \sum_{k'}
 \end{textstyle}
 m_{k'}^i\eee^{-\tilde t_{k'}^i}\1_{\tilde t_{k'}^j\in[r_1,t]} +[-\d,\d] \Big|\FF_{r_1}\Bigr).
 \eea
Due to the independence of $(\tilde t_k^j,k\in\N)$ and $(\tilde t_{k'}^i,k'\in\N)$ we can proceed as with \eqv(map.20) in the proof of Lemma \thv(lem.prop.2) 
to make \eqv(map.43) as small as desired, independently on the value of $\g(x_{\ell'}(r_1))-\g(x_\ell(r_1))$ by choosing $\d$ small enough.
Collecting all terms, we see that  \eqv(seltsam.1) is bounded by
\be\Eq(seltsam.5)
 \E\left[ \sum_{\ell\neq \ell'=1}^{n(r_1)} \P\left(   \dots \big |\FF_{r_1}\right)\right]\leq 4\e''  \E\left[ \sum_{\ell\neq \ell'=1}^{n(r_1)} 1\right] \leq 4\e'' K\eee^{2r_1}.
 \ee
Choosing $\e''$  and $\e'$ small enough,  this yields the assertion of Lemma  \thv(lem.prop.3).
\end{proof}

%
%
%
%


\section{The $q$-thinning}
The proof of the convergence of $\sum_{i=1}^{n(t)}\d_{(\g(x_i(t)),x_i(t)-m(t))}$ comes in two main steps. In a first step, we show that the points of the local extrema converge to the desired Poisson point process. To make this precise, we work with the concept of thinning classes that was already introduced in \cite{ABK_P}. We repeat the construction here for completeness and introduce the corresponding notation.

Assume here and in the sequel that the particles at time $t$ are labeled in decreasing order
\be\Eq(thin.1)
x_1(t)\geq x_2(t)\geq\dots\geq  x_{n(t)}(t),
\ee
and set $\bar x_k(t)\equiv x_k(t)-m(t)$.
Let 
\be\Eq(thin.2)
\bar Q(t)=\{\bar Q_{i,j}(t)\}_{i,j\leq n(t)}\equiv \{t^{-1}Q_{i,j}(t)\}_{i,j\leq n(t)},
\ee
where 
\be\Eq(thin.3)
Q_{i,j}(t)=\sup\{s\leq t:x_i(s)=x_j(s)\} =d(u^i(t),u^j(t)).
\ee
$(\mathcal{E}(t),\bar Q(t))$ admits the following thinning. For any $q\geq 0$ the 
following is true:
If $ \bar Q_{i,j}(t)\geq q$ and $\bar Q_{j,k}(t)\geq q$,  then $\bar Q_{i,k}(t)\geq q$. Therefore,
the sets $\{i,j\in \{1,\dots, n(t)\}: \bar Q_{i,j}(t)\geq q\}$ form a partition of the set 
$\{1,\dots,n(t)\}$ into equivalence classes. 
We select  the maximal particle of each equivalence class as representative in the 
following recursive manner: 
\bea\Eq(thin.5)
&& i_1=1 \nonumber\\
&& i_k=\min\{j\geq i_{k-1}:\bar Q_{i,j}(t)< q,\;\forall i\leq k-1\},
\eea
if such an $j$ exists. If no such $j$ exists, we denote $k-1=n^*(t)$ and terminate the procedure.
The $q$- thinning process of $(\mathcal{E}(t),\bar Q(t))$, denoted by 
$\mathcal{E}^{(q)}(t)$ is defined by
\be\Eq(thin.6)
\mathcal{E}^{(q)}(t)=\sum_{k=1}^{n^*(t)}\d_{\bar x_{i_k}(t)}.
\ee

\section{Extended convergence of thinned point process}
For $r_d\in\R_+$ and $t>3r_d$ consider the thinned process $\EE^{(r_d/t)}(t)$. Observe that, for $R_t=m(t)-m(t-r_d)-\sqrt 2 r_d=o(1)$,  we have
\be\Eq(prop.pois.1)
\EE^{(r_d/t)}(t)\stackrel{D}{\equiv}\sum_{j=1}^{n(r_d)}
\d_{x_j(r_d)-\sqrt{2} r_d+M_j(t-r_d)-R_t}
\ee
where $M_j(t-r_d)\equiv \max_{k\leq n^{(j)}(t-r_d)}x_k^{(j)}(t-r_d)-m(t-r_d)$ and $x^{(j)}$ are independent BBM's (see (3.15) in \cite{ABK_P}).
Then
\begin{proposition}\TH(prop.pois)
 Let $\mathcal{E}^{(r_d/t)}(t)$ and $n^*(t)$ be defined in \eqv(thin.6) for $q=r_d/t$. Then
 \be\Eq(pois.1)
 \lim_{r_d\uparrow\infty}\lim_{t\uparrow \infty}\sum_{k=1}^{n^*(t)}\d_{(\g(x_{i_k}(t)),\bar x_{i_k}
 (t))}\stackrel{D}{=}\sum_i \d_{(q_i,p_i)}\equiv \wh \EE,
 \ee
 where $(q_i,p_i)_{i\in\N}$ are the points of the Cox process $\wh\EE$  
 with intensity 
 measure $Z(dv)\times C\eee^{-\sqrt 2 x}dx$ with  the random measure $Z(dv)$ defined in 
  \eqv(Z.2).
 Moreover,
 \be\Eq(prop.pois.2)
 \lim_{r\uparrow\infty}
\lim_{r_d\uparrow\infty}\sum_{j=1}^{n(r_d)}\d_{(\gamma(x_j(r)),x_j(r_d)-\sqrt{2} r_d+M_j)}
\stackrel{D}{=}\wh \EE,
 \ee 
 where $M_j$ are i.i.d with law $\omega$ defined in \eqv(extremal.1.1).
\end{proposition}

The proof of Proposition \thv(prop.pois) relies on Proposition \thv(lem.Z) which we now prove.

\begin{proof}[Proof of Proposition \thv(lem.Z)]
 For $v,r\in\R_{+}$ fixed, the process  $Z(v,r,t)$ defined in \eqv(Z.1) is a martingale in 
 $t>r$ (since $Z(\infty,r,t) $ is the derivative martingale and $\1_{\g(x_i(r))\leq v}$ does 
 not depend on $t$). 
 To see that $Z(v,r,t)$ converges a.s. as $t\uparrow\infty$, 
 note that 
 \bea\Eq(sellke.1)
 Z(v,r,t)&=&\sum_{i=1}^{n(r)} \1_{\g(x_i(r))\leq v}\eee^{\sqrt 2(x_{i}(r)-\sqrt 2r)}
 \Biggl( \left(\sqrt 2 r-x_{i}(r)\right)\sum_{j=1}^{n^{(i)}(t-r)}
 \eee^{\sqrt 2(x^{(i)}_j(t-r)-\sqrt 2(t-r))}\nonumber\\
&&\quad+ \sum_{j=1}^{n^{(i)}(t-r)} \left(\sqrt 2(t-r)-x_j^{(i)}(t-r)\right)
 \eee^{\sqrt 2(x^{(i)}_j(t-r)-\sqrt 2(t-r))}\Biggr)\nonumber\\
 &=&
 \sum_{i=1}^{n(r)} \1_{\g(x_i(r))\leq v}\eee^{\sqrt 2(x_{i}(r)-\sqrt 2r)}
 \left(\sqrt 2 r-x_{i}(r)\right) Y^{(i)}_{t-r}\nonumber\\
 &&\quad+ \sum_{i=1}^{n(r)} \1_{\g(x_i(r))\leq v}
 \eee^{\sqrt 2(x_{i}(r)-\sqrt 2r)} Z^{(i)}_{t-r}.
 \eea
 Here $Z^{(i)}_t, i\in \N$, are iid copies of the derivative martingale, and $Y^{(i)}_t, i\in\N$, are
 iid copies of the McKean martingale,
 \be\Eq(macky.1)
 Y_t\equiv \sum_{i=1}^{n(t)}  \eee^{\sqrt 2(x^{(i)}_j(t)-\sqrt 2t)}.
 \ee
  Lalley and Sellke proved in \cite{LS} that
 $\lim_{t\uparrow \infty} Y_t=0$, a.s. while $\lim_{t\uparrow\infty} Z_t=Z$  exists a.s. and  is a non-trivial 
 random variable. This implies that
 \be\Eq(sellke.2)
 \lim_{t\uparrow\infty}Z(v,r,t)\equiv Z(v,r)=\sum_{i=1}^{n(r)} 
 \eee^{\sqrt 2(x_{i}(r)-\sqrt 2r)} Z^{(i)}\1_{\g(x_i(r))\leq v},
 \ee
 where $Z^{(i)}$, $i\in \N$ are iid copies of $Z$.
To show that $Z(v,r)$ converges, as $r\uparrow\infty$, we go back to \eqv(Z.1). 
Note that, for fixed $v$,  $\1_{\g(x_i(r))\leq v}$ is 
 monotone decreasing in $r$. On the other hand, Lalley and Sellke have shown that
 $\min_{i\leq n(t)}\left(\sqrt 2 t-x_i(t)\right) \to +\infty$, almost surely, as $t\uparrow \infty$.
 Therefore, the part of the sum in \eqv(Z.1) that involves negative terms (namely
 those for which $x_i(t)>\sqrt 2 t$) converges to zero, almost surely. The remaining part
 of the sum is decreasing in $r$, and this implies that the limit, as $t\uparrow\infty$,
 is monotone decreasing almost surely. 
Moreover,  $0\leq Z(v,r)\leq Z$, a.s.,  where $Z$ is the almost sure limit of the derivative martingale. Thus  $\lim_{r\uparrow \infty}Z(v,r)\equiv Z(v)$ exists. 
Finally,  $0\leq Z(v)\leq Z $ and $Z(v)$ is an increasing function of $v$ because $Z(v,r)$
 is increasing in $v$, a.s., for  each $r$.

 To show that $Z(du)$ is nonatomic, fix $\e,\d>0$ and let $D\subset \R$ be compact.
By Lemma \thv(lem.prop.3) there exists  $r_1(\e,\d)$ such that, 
for all $r>r_1(\e,\d)$ and $t>3r$, 
 \be\Eq(Z.11)
 \P\left(\exists_{i,j\leq n(t)}: d(x_i(t),x_j(t))\leq r, x_i(t),x_j(t)\in m(t)+D, \vert \g(x_i(t))-\g(x_j(t))\vert\leq \d\right)<\e.
 \ee
 Rewriting \eqv(Z.11) in terms of the thinned process $\mathcal{E}^{(r/t)}(t)$  gives
 \be\Eq(Z.13)
 \P\left(\exists_{i_k,i_{k'}}: \bar x_{i_k}(t),\bar x_{i_{k'}}(t)\in m(t)+D,|\g(\bar x_{i_k}(t))-
 \g(\bar x_{i_{k'}}(t))|\leq \d\right)\leq \e.
 \ee
 Assuming, for the moment, that $\mathcal{E}^{(r/t)}(t)$ converges as claimed in 
 Proposition \thv(prop.pois), this implies that, for any $\e>0$, for small enough $\d>0$, 
 \be\Eq(Z.14)
 \P\left(\exists\d>0: \exists i\neq j: |q_i-q_j|<\d \right)<\e.
 \ee
 This could not be true if $Z(du)$ had an atom. This proves Proposition \thv(lem.Z) 
provided we can show convergence of $\mathcal{E}^{(r/t)}(t)$. 
 \end{proof}
 
The proof of Proposition \thv(prop.pois) uses the properties of the map $\g$ obtained in Lemma \thv(lem.prop.1) and \thv(lem.prop.2). In particular, we use that, in the limit as $t\uparrow \infty$, the image of the extremal particles under $\g$ converges and that essentially no particle is mapped too close to the boundary of any given compact set. Having these properties at hand we can use the same  procedure as in the proof of Proposition 5 in \cite{ABK_P}. Finally, we use Proposition \thv(lem.Z) to deduce Proposition \thv(prop.pois).

\begin{proof}[Proof of Proposition \thv(prop.pois)]
 We show the convergence of the Laplace functionals. Let $\phi:\R_+\times\R\to \R_+$ be a measurable function with compact support. For simplicity we start by looking at simple functions of the form
 \be\Eq(pois.2)
 \phi(x,y)=\sum_{i=1}^N a_i\1_{A_i\times B_i}(x,y),
 \ee
 where  $A_i=[\underline{A}_i,\overline{A}_i]$ and $B_i=[\underline{B}_i,\overline{B}_i]$,
 for  $N\in \N$, $i=1,\dots,N$, $a_i,\underline{A}_i,\overline{A}_i\in\R_+$, and 
 $\underline{B}_i, \overline{B}_i\in\R$. The extension to general functions $\phi$ then 
 follows by monotone convergence.
For such $\phi$,  we consider the Laplace functional
 \be\Eq(pois.5)
\Psi_t(\phi)\equiv  \E\left[\exp\left(-\sum_{k=1}^{n^*(t)}\phi\left(\g(x_{i_k}(t)),\bar x_{i_k}(t)\right)
\right)\right]. 
 \ee
 The idea is that the function $\g$ only depends on the early branchings of the particle. 
 To this end we insert the identity
 \be\Eq(pois.6)
 1=\1_{\mathcal{A}^{\g}_{r,t}(\supp_y \phi)}+\1_{\left(\mathcal{A}^{\g}_{r,t}(\supp_y \phi)\right)^c}
 \ee
 into \eqv(pois.5), where  $\AA^\g_{r,t}$ is defined  in  \eqv(map.4.1),
 and by $\supp_y \phi$ we mean the support of $\phi$ with respect to 
 the second variable. 
 By Lemma \thv(lem.prop.1) we have that, for all $\e>0$, there exists $r_\e$ such that, 
 for all $r>r_\e$, 
 \be\Eq(pois.7)
 \P\left(\left(\mathcal{A}^{\g}_{r,t}(\supp_y \phi)\right)^c\right)<\e,
 \ee
 uniformly in $t>3r$. Hence it suffices to show the convergence of
 \be\Eq(pois.8)
 \E\left[\exp\left(-\sum_{k=1}^{n^*(t)}\phi\left(\g(x_{i_k}(t)),\bar x_{i_k}(t)\right)\right)\1_{\mathcal{A}^{\g}_{r,t}(\supp_y \phi)}\right].
 \ee
  We introduce yet  another identity into \eqv(pois.8), namely 
  \be \Eq(pois.9)
  1= \1_{\bigcap_{i=1}^N\left(\mathcal{B}_{r,t}^{\g}(\supp_y \phi, \underline{A}_i)\cap \mathcal{B}_{r,t}^{\g}(\supp_y \phi, \overline{A}_i )\right) } 
   +\1_{\left(\bigcap_{i=1}^N\left(\mathcal{B}_{r,t}^{\g}(\supp_y \phi, \underline{A}_i)\cap \mathcal{B}_{r,t}^{\g}(\supp_y \phi, \overline{A}_i)\right)\right)^c},
  \ee
  where we use the shorthand notation
   $ \mathcal{B}_{r,t}^{\g}(\supp_y \phi, \overline{A}_i)\equiv \mathcal{B}_{r,t}^{\g}(\supp_y 
   \phi, \overline{A}_i,\eee^{-r/2})$ 
  (recall \eqv(map.15.1)). By Lemma \thv(lem.prop.2), for all 
  $\e>0$ there exists  $\bar{r}_{\e}$ such that, for all $r>\bar r_\e$ and uniformly in $t>3r$,
  \be\Eq(pois.10)
  \P\left(\left(\cap_{i=1}^N\left(\mathcal{B}_{r,t}^{\g}(\supp_y \phi, \underline{A}_i)\cap \mathcal{B}_{r,t}^{\g}(\supp_y \phi, \overline{A}_i)\right)\right)^c\right)<\e.
  \ee
 Hence we only have to show the convergence of 
 \be \Eq(pois.11)
  \E\left[\exp\left(-\sum_{k=1}^{n^*(t)}\phi\left(\g(x_{i_k}(t)),\bar x_{i_k}(t)\right)\right) 
  \1_{\mathcal{A}_{r,t}^{\g}(\supp_y \phi)\cap\left( \bigcap_{i=1}^N
  \left(\mathcal{B}_{r,t}^{\g}(\supp_y \phi, \underline{A}_i)\cap \mathcal{B}_{r,t}^{\g}
  (\supp_y \phi, \overline{A}_i)\right) \right)}\right].
 \ee
 Observe that on the event in the indicator function in the the last line
  the following holds: If, for any $i\in\{1,\dots,N\}$, $\g(x_{k}(t))\in[\underline{A}_i,\overline{A}_i]$ and $\bar x_k(t)\in\supp_y \phi$
 then also $\g(x_k(r))\in [\underline{A}_i,\overline{A}_i]$, and vice versa. 
 Hence \eqv(pois.11) is equal to
 \be \Eq(pois.12)
 \E\Bigg[\exp\left(-\sum_{k=1}^{n^*(t)}\phi\left(\g(x_{i_k}(r)),\bar x_{i_k}(t)\right)\right) 
 \1_{\mathcal{A}_{r,t}^{\g}(\supp_y \phi)\cap\left( \bigcap_{i=1}^N\left(\mathcal{B}_{r,t}^{\g}(\supp_y \phi, \underline{A}_i)\cap \mathcal{B}_{r,t}^{\g}(\supp_y \phi, \overline{A}_i)\right) \right)}\Bigg].
 \ee
 Now we apply again Lemma \thv(lem.prop.1) and Lemma \thv(lem.prop.2) to see that the quantity in \eqv(pois.12) is 
  equal to
 \be\Eq(pois.13)
 \E\left[\exp\left(-\sum_{k=1}^{n^*(t)}\phi\left(\g(x_{i_k}(r)),\bar x_{i_k}(t)\right)\right)\right]
 +O(\e).
 \ee
Introducing a conditional expectation given $\FF_{r_d}$,  we get (analogous to (3.16) in 
\cite{ABK_P}) as $t\uparrow \infty$ that \eqv(pois.13) is equal to
\bea\Eq(pois.13.2)
&&\lim_{t\uparrow\infty}\E\left[\exp\left(-\sum_{k=1}^{n^*(t)}\phi\left(\g(x_{i_k}(r)),\bar x_{i_k}(t)\right)\right)\right]\\\nonumber
&&=\lim_{t\uparrow\infty}\E\left[\prod_{j=1}^{n(r_d)}
\E\left[\eee^{-\phi(\g(x_j(r)),x_j(r_d)-m(t)+m(t- r_d)+
\max_{i\leq n^{(j)}(t-r_d)} x^{(j)}_i(t-r_d)-
m(t-r_d) )}  \big|\FF_{r_d}\right]\right]
\\\nonumber
&&=\E\left[\prod_{j=1}^{n(r_d)}\E\left[\eee^{-\phi(\g(x_j(r)),x_j(r_d)-\sqrt 2 r_d+M)}
\big|\FF_{r_d}\right]\right],
\eea
where $M$ is the limit of the centred maximum of BBM, whose distribution is 
given in  \eqv(extremal.1.1). Note that $M$ is independent of $\FF_{r_d}$. 
The last expression is completely analogous to Eq. (3.17) in \cite{ABK_P}. Following the 
analysis of this expression up to Eq. (3.25) in  \cite{ABK_P}, 
we find that \eqv(pois.13.2) is equal to 
 \be\Eq(pois.14)
c_{r_d} \E\Bigl[\exp\Bigl(-C\sum_{j\leq n(r_d)}y_j(r_d)\eee^{-\sqrt{2}y_j(r_d)}\sum_{i=1}^N (1-\eee^{a_i})\1_{A_i}(\g(x_j(r)))\bigl(\eee^{-\sqrt{2}\; \underline{B}_i}-\eee^{-\sqrt{2}\;\overline{B}_i}\bigr)\Bigr)\Bigr],
 \ee
 where $y_j(r_d)=x_j(r_d)-\sqrt 2 r_d$, $\lim_{r_d\uparrow \infty}c_{r_d}=1$, and 
 $C$ is the constant from \eqv(extremal.1.1).
 Using Proposition \thv(lem.Z) \eqv(pois.14) is in the limit as $r_d\uparrow \infty$ and $r\uparrow \infty$ equal to
 \bea\Eq(pois.16)
&& \E\left[\exp\left(- C\sum_{i=1}^N (1-\eee^{a_i})\bigl(\eee^{-\sqrt 2 \underline{B}_i}-\eee^{-\sqrt 2\;
\overline{B}_i}\bigr)\right)(Z(\overline{A_i})-Z(\underline {A_i}))\right]\\\nonumber
&&= \E\left[\exp\left(\int \left(\eee^{-\phi(x,y)}-1\right) Z(dx)\sqrt 2C\eee^{-\sqrt 2y}dy\right)
 \right].
 \eea
 This is the Laplace functional of the process $\wh \EE$, 
 which proves Proposition \thv(prop.pois).
\end{proof}
To prove Theorem \thv(thm.extend) we need to combine 
  Proposition \thv(prop.pois) with the
 results on the genealogical structure of the extremal particles of BBM obtained in 
 \cite{ABK_G} and the convergence  of the decoration point process $\D$ (see e.g. Theorem 2.3 of \cite{ABBS}).
 
\begin{proof}[Proof of Theorem \thv(thm.extend)]
For $x_{i_k}(t) \in \supp\left(\mathcal{E}^{(r_d/t)}(t)\right)$ define the process of recent
relatives by 
\be\Eq(extend.3)
\D_{t,r}^{(i_k)}=\d_0+\sum_{j:\t_j^{i_k}>t-r}\mathcal{N}_j^{i_k},
\ee
where $\t_j^{i_k}$ are the branching times along the path $s\mapsto x_{i_k}(s)$ 
enumerated backwards in time and $\mathcal{N}_j^{i_k}$ the point measures of particles 
whose ancestor was born at $\t_j^{i_k}$. In the same way let $\D_r^{(i_k)}$ be 
independent copies of $\D_r$ which is defined as (recall \eqv(extremal.3))
\be\Eq(deltar.1)
\D_r\equiv \lim_{t\uparrow\infty}\sum_{i=1}^{n(t)} \1_{d\left(\tilde x_i(t),
\tilde x_{\arg\max_{j\leq n(t)}\tilde x_j(t)}(t)\right)\geq t-r}\,
 \d_{\tilde x_i(t)-\max_{j\leq n(t)}\tilde x_j(t)}
\ee
conditioned on $\max_{j\leq n(t)}\tilde x_j(t)\geq \sqrt{2}t$, the point measure obtained from $\D$ 
by only keeping particles that branched of the maximum after  time $t-r$ (see the 
backward description of $\D$ in \cite{ABBS}). By Theorem 2.3 of \cite{ABBS} we 
have that (the labelling $i_k$ refers to the thinned process $\EE^{(r_d/t)}(t)$)
\be\Eq(extend.4)
\left(x_{i_k}(r_d)-\sqrt{2} r_d+M_{i_k}(t-r_d),\D_{t,r_d}^{(i_k)}\right)_{1\leq k\leq n^*(t)} \Rightarrow 
\left(x_j(r_d)-\sqrt{2} r_d+M_{j},\D_{r_d}^{(j)}\right)_{j\leq {n(r_d)}},
\ee
as $t\uparrow \infty$,
where $M_j $ are independent copies of $M$ with law $\omega$ (see \eqv(extremal.1.1)). Moreover, $\D_{r_d}^{(j)}$ is independent of $(M_j)_{j\leq n(r_d)}$. 
 Looking now at the the Laplace functional for the complete point process $\wt \EE_t$,
 \be\Eq(all.1)
 \wt\Psi_t(\phi)\equiv \E \left[\eee^{-\int \phi(x,y)\wt\EE_t(dx,dy)}\right],
 \ee for $\phi$ as in \eqv(pois.2),  and doing the same manipulations as in the proof of 
 Proposition \thv(prop.pois), shows that 
 \be\Eq(extend.5.1)
\wt \Psi_t(\phi)=\E\left[\exp\left(-\sum_{k=1}^{n(t)} \phi\left(\g(x_{k}(r)),\bar x_{k}(t)\right)
\right) \right]
 +O(\e).
 \ee
 Denote by $\mathcal{C}_{t,r}(D)$ the event
 \be\Eq(extend.2)
 \mathcal{C}_{t,r}(D)=\{\forall i,j\leq n(t) \mbox{ with }x_i(t),x_j(t)\in D+m(t) \mbox{: }d(x_i(t),x_j(t))\not \in (r,t-r)\}.
 \ee
 By Theorem 2.1 in \cite{ABK_G} we know that, for each $D\subset \R$ compact, 
 \be\Eq(extend.1)
 \lim_{r\uparrow\infty}\sup_{t>3r}\P\left(\left(\mathcal{C}_{t,r}(D)\right)^c\right)=0.
 \ee
 Hence by introducing $1=\1_{\left(\mathcal{C}_{t,r}(\supp_y \phi)\right)^c}+\1_{\mathcal{C}_{t,r}(\supp_y \phi) }$ into \eqv(extend.5.1), we obtain that 
 \be\Eq(extend.5) 
\wt\Psi_t(\phi)=\E\left[\eee^{-\sum_{k=1}^{n^*(t)}\left(\phi\left(\g(x_{i_k}(r)),\bar x_{i_k}(t)\right)+\sum_j \phi\left(\g(x_{i_k}(r)),\bar x_{i_k}(t)+
\left(\D_{t,r_d}^{(i_k)}\right)_j\right)\right)}\right]
+O(\e) ,
\ee
 where $\left(\D_{t,r_d}^{(i_k)}\right)_j$ are the atoms of $\D_{t,r_d}^{(i_k)}$. Hence it suffices to show that
 \be \Eq(extend.6)
 \sum_{k=1}^{n^*(t)} \sum_{\ell}\d_{(\g(x_{i_k}(r)),\bar x_{i_k}(t))+(0,\left(\D_{t,r_d}^{(i_k)}\right)_\ell)  }
 \ee
 converges weakly when first taking the limit $t\uparrow \infty$ and then the limit $r_d\uparrow \infty$ and finally $r\uparrow \infty$. 
But by \eqv(extend.4),
 \be \Eq(extend.7)
 \lim_{t\uparrow \infty}\sum_{k=1}^{n^*(t)} \sum_{\ell}\d_{(\g(x_{i_k}(r)),\bar x_{i_k}(t))+\left(0,\left(\D_{t,r_d}^{(i_k)}\right)_\ell\right)  }=
 \sum_{j=1}^{n(r_d)} \sum_{\ell}\d_{(\g(x_{j}(r)), x_{j}(r_d)-\sqrt 2 r_d+M_j)+\left(0,\left(\D_{r_d}^{(j)}\right)_\ell\right)}.  
 \ee
  The limit as first $r_d$ and then $r$ tend to infinity of the process on the right-hand side 
  exists and is equal to $\wt\EE$ by Proposition  \thv(prop.pois) (in particular \eqv(prop.pois.2)). This 
  concludes the proof of Theorem \thv(thm.extend).  
\end{proof}



\end{document}